\documentclass[ECP]{ejpecp} 

\SHORTTITLE{The Contact Process on Periodic Trees} 

\TITLE{The Contact Process on Periodic Trees\thanks{Partially supported by DMS grant DMS 1505215 from the Probability Program. 
This original version of this paper \cite{melody} was produced in a DOMath project at Duke University in the summer of 2018 with undergraduates Yufeng Jiang, Remy Kassem, Grayson York, and Brandon Zhao.}}

\AUTHORS{%
  Xiangying Huang
\footnote{
    \EMAIL{zoe@math.duke.edu}}
  \and
 Rick Durrett \footnote{Dept. of Math,  Duke University, Box 90320, Durham NC 27708-0320
 \BEMAIL{rtd@math.duke.edu}
 }}


\AMSSUBJ{60K35} 

\SUBMITTED{August, 2019} 
\ACCEPTED{July, 2022} 

\VOLUME{0}
\YEAR{2012}
\PAPERNUM{0}
\DOI{vVOL-PID}

\ABSTRACT{A little over 25 years ago Pemantle  \cite{Pem92} pioneered the study of the contact process on trees, and showed that on homogeneous trees the critical values $\lambda_1$ and $\lambda_2$ for global and local survival were different. He also considered trees with periodic degree sequences, and Galton-Watson trees. Here, we will consider periodic trees in which the number of children in successive generation is $(n,a_1,\ldots, a_k)$ with $\max_i a_i \le Cn^{1-\delta}$ and $\log(a_1 \cdots a_k)/\log n \to b$ as $n\to\infty$. We show that the critical value for local survival is asymptotically $\sqrt{c (\log n)/n}$ where $c=(k-b)/2$. This supports Pemantle's claim that the critical value is largely determined by the maximum degree, but it also shows that the smaller degrees can make a significant contribution to the answer. }

\newcommand{\ep}{\epsilon}
\newcommand\ZZ{\mathbb{Z}}
\newcommand\TT{\mathbb{T}}

\newcommand\beq{ \begin{equation} }
\newcommand\eeq{ \end{equation} }
\newcommand\mn{\medskip\noindent}

\begin{document}



\section{Introduction}

The contact process can be defined on any graph as follows: occupied sites become vacant at rate 1, while
vacant sites become occupied at rate $\lambda$ times the number of occupied neighbors. 
Harris \cite{Harris} introduced the contact process on $\ZZ^d$ in 1974. It has been extensively studied, see Liggett \cite{Lig99} for a summary of
most of what is known. 

Pemantle \cite{Pem92} began the study of contact processes on trees.
Let $\xi_t$ be the set of occupied sites at time $t$ and use $\xi^0_t$ to denote the process with $\xi^0_0 = \{0\}$
where 0 is the root of the tree. His main new result was that the process had two phase transitions:
\begin{align*}
\lambda_1 &= \inf\{ \lambda : P( \xi^0_t \neq \emptyset \hbox{ for all $t$}) > 0 \} \\
\lambda_2&  = \inf\{ \lambda :  \liminf_{t\to\infty} P( 0 \in \xi^0_t) > 0 \}.
\end{align*}
Let $\TT_d$ be the tree in which each vertex has $d+1$ neighbors. When $d=1$, $\TT_1=\ZZ$, so we restrict our attention to $d \ge 2$. Pemantle showed that $\lambda_1< \lambda_2$ when $d\ge 3$ by getting upper bounds on $\lambda_1$ and lower bounds on $\lambda_2$. Liggett \cite{Lig96} proved that in $d=2$ $\lambda_1 < 0.605 < 0.609 < \lambda_2$ to settle the last case.  In \cite{Stacey96} Stacey gave an elegant proof that on $\TT_d$ and a number of other graphs we have $\lambda_1< \lambda_2$.  

Pemantle also considered periodic trees and Galton-Watson trees. In the special case that the number of children alternates between $a$ and $b$ he showed
$$
1/(\sqrt{a} + \sqrt{b}) \le \lambda_2.
$$
He did not give the details of the proof, but this can easily be proved using Lemma 3.1 in Pemantle and Stacey \cite{PemSta01}, which gives a formula for the critical value for local survival for branching random walk on a general graph. Pemantle also showed, see the first sentence after (7) on page 2103,
that for a general period $k$ tree with degree sequence $(a_1,a_2,\dots,a_k)$ the critical value 
$\lambda_2 \le C/(a_1 \cdots a_k)^{1/2k}. $
When $k=2$ the bound is $C/(ab)^{1/4}$.
When $a=b$, the upper and lower bounds differ by a factor of 2. However when $a=1$ and $b=n$ 
\beq
1/(\sqrt{a} + \sqrt{b}) = 1/(1+\sqrt{n}) \qquad C/(ab)^{1/4} = C/n^{1/4}.
\label{twpbds} 
\eeq

As Pemantle notes on page 2103, the upper and lower bounds are different orders of magnitude. He continues with 
``Which of these asymptotics for $\lambda_2$ is sharp if either? The somewhat surprising answer is that the lower bound is
sharp even though the geometric mean $(a_1 \cdots a_k)^{1/k}$  is clearly a better representative for the growth rate of the tree.'' 
The next result shows that the lower bound is more accurate than the upper bound, but it is not quite sharp.  

\begin{theorem} \label{ub1n}
On the $(1,n)$ tree, as $n \to \infty$ the critical value 
$$
\lambda_2 \sim  \sqrt{c_2(\log n)/n}\quad\hbox{where $c_2=1/2$}.
$$
\end{theorem}

\mn
On page 2103, Pemantle says that ``for reasonably regular non-homogeneous trees the critical value is determined by $M$ the maximum number of children and
is at most $rM^{-1/2}$ where $r$ is a logarithmic measure of how far apart vertices with $M$ children are from each other.''
The next result confirms his intuition about the importance of the maximum degree but also shows that the lower degree vertices can have a significant
influence on $\lambda_2$.

\begin{theorem} \label{akub2}
Consider the $(n,a_1,a_2, \ldots, a_k)$ tree with $\max_i a_i \le  Cn^{1-\delta}$ for some positive $C, \delta$ and suppose
$$
b=\lim_{n\to\infty} \frac{\log (a_1 a_2 \cdots a_k)}{\log n}.
$$
As $n \to \infty$ the critical value $\lambda_2 \sim  \sqrt{c_k\log n/n}$ where 
$c_k = (k-b)/2$.
\end{theorem}

\noindent
Theorem \ref{ub1n} is the special case $k=1$, $a_1=1$, so it suffices to prove Theorem \ref{akub2}.

For readers (and referees) who think that proving results on periodic trees is trivial we state two open questions

\mn
{\bf Problem 1.} {\it Find the asymptotic behavior of $\lambda_1$ on the $(n,a_1,a_2, \ldots, a_k)$ tree as $n\to\infty$.}

\mn
Hint:The answer is $\sqrt{c'_k/(\log n)/n}$ when $c_k' = (k+1)/2 - (b+1)>0$ but when $c'_k<0$ $\lambda_1$ is much smaller. The asymptotics tell us that $\lambda_1 < \lambda_2$ for large $n$. Stacey's results in \cite{Stacey96} imply that $\lambda_1 < \lambda_2$ on period two trees, but it is open to prove (see page 1725 in \cite{Stacey96}).

\mn
{\bf Problem 2.} {\it Prove that $\lambda_1 < \lambda_2$ on the (2,3,4) tree.}

\subsection{The survival time of the contact process on finite trees}

Let $\rho$ be the root of the periodic tree $(n,a_1,\dots,a_k)$. Truncating the periodic tree at height $k$ gives a subgraph $S_k=\{ x: d(\rho,x) \leq k\}$, where $d$ is the distance on the tree. A vertex $x\in S_k$ at distance $i$ from the center $\rho$ is said to be in the set $L_i$ (``level $i$"). In $S_k$, vertices on level $1 \le i < k$ have $a_i$ children  while vertices on level $k$ are leaves, i.e., they have no children. When the context is clear we also write $S_k=(n,a_1,\dots,a_k)$, where the sequence corresponds to the offspring number on each level. When we delete the root from $S_k$ we end up with $n$ subtrees $\{T_{k,i}\}_{i=1}^n$ with $T_{k,i}=(a_1,\dots, a_k)$.

The first step in the proof of Theorem \ref{akub2} is to prove in Section 2  an upper bound on the survival time of contact process on $S_k$.

\begin{theorem}\label{ubsurv}
Suppose $\lambda=\sqrt{ c(\log n)/n}$ where $c>0$. Let $\tau_k$ be the survival time of contact process on $S_k=(n,a_1,\dots,a_k)$ starting from all sites occupied where $\max_i a_i \le C n^{1-\delta}$ for some positive $C$ and $\delta$. For any $\ep>0$, when $n$ is sufficiently large
$$E \tau_k \leq C_0(\log n)e^{(1+\ep)\lambda^2 n},
$$
where $C_0$ is some positive constant depending on $k$ but not on $C, \delta$.
\end{theorem}

When $k=1$, $S_k$ reduces to the star graph. In this case the result holds for all $\lambda$ (see \cite{melody} for details). Lemma \ref{ubsurv} gives the only upper bound we know of for the survival time for the contact process on the star. 
There are many lower bounds for the survival time on stars. See Theorem 4.1 in \cite{Pem92}, Lemma 5.3 in \cite{BBCS}, and Lemma 1.1 in \cite{ChaDur}. These bounds can be used to show that the critical value for prolonged survival of the contact process on some random graphs is 0, but to identify the asymptotics for the critical value on the $(n,a_1,\dots,a_k)$ tree, we need a more precise result on the survival time on the star graph

\mn
\begin{theorem} \label{lbsurv}
 Let $L = (1-4\delta)\lambda n$ with $\delta>0$ . If $\eta>0$ is small then
$$
P_{L,1} \left( T_{0,0} \ge \frac{1}{\lambda^2 n} e^{(1-\eta)\lambda^2 n} \right) \to 1 \qquad \hbox{as $n\to\infty$.}
$$
\end{theorem}

\noindent
Combining this with Theorem \ref{ubsurv} shows that if $\lambda^2 n \to\infty$ the survival time on the star is $\exp( (1+o(1) )\lambda^2 n )$.

\section{Upper bound on survival times on $S_k$}

\subsection{Equilibrium on $T_k^*$}
To prepare for the proof of Theorem\ref{ubsurv}, we first consider the contact process on $S_k$ with the center $\rho$ permanently occupied. Let $T_k=(a_1,\dots,a_k)$ denote a generic subtree of $\rho$ and let $T^*_k$ be $T_k$ with $\rho$ attached to the root and with $\rho$ permanently occupied. Since the center $\rho$ is always occupied, the contact process on $S_k$ can be simply treated as $n$ independent contact processes on the $T^*_{k,i}$'s. As $\rho$ is always occupied, the contact process $\xi_t$ on $T^*_k$ has a stationary distribution $\xi_{\infty}$. 

We use $x_i$ to denote a generic vertex in $L_i$. To compute the occupancy probability for a vertex $x_i\in L_i$ we need the notion of the dual process $\zeta_t$ of $\xi_t$. To construct $\xi_t$ by graphical representation, we assign a Poisson process $N_x$ of rate 1 to each vertex $x\in T^*_k\backslash \{\rho\}$ and a Poisson process $N_{(x,y)}$ of rate $\lambda$ to each ordered pair of vertices that are joined by an edge of $T^*_k$. The dual process $\zeta_t$ is constructed by looking at the dual path on the same graphical representation we used to construct $\xi_t$. (See Liggett \cite{Lig99} for an account of graphical representation and duality.) It follows from duality that
$$
P(x_i\in \xi_t)=P(\rho\in \zeta^{x_i}_s \text{ for some }s\leq t)\leq P(\rho \in \zeta^{x_i}_s \text{ for some }s\geq 0).
$$
Letting $t\to\infty$ gives $P(x_i\in \xi_\infty)\leq P(\rho \in \zeta^{x_i}_s \text{ for some }s\geq 0),$
That  is, if $x_i\in \xi_\infty$ then the dual contact process $\zeta^{x_i}_t$ starting from $x_i$ has to reach $\rho$ at some time. If we have a dual path of length $i+2m$ from $x_i$ to $\rho$ then $i+m$ steps will be toward $\rho$ and $m$ steps away. Let $(y_0,y_1,\dots,y_{i+2m})$ denote a path from $x_i$ to $\rho$ with $y_0=x_i$ and $y_{i+2m}=\rho$. To produce a particle at $\rho$, we need a birth from $y_j$ to $y_{j+1}$ to occur before the particle at $y_j$ dies for all $j=0,\dots i+2m-1$. So the expected number of particles produced at $\rho$ by this path is 
$$
\left(\frac{\lambda}{1+\lambda}\right)^{i+2m}\leq \lambda^{i+2m}.
$$ 
If we let $d=Cn^{1-\delta}$ so that $a_i\leq d$ for all $i=1,\dots, k$, then the expected number of particles $N_{x_i,\rho}$ that reach $\rho$ has 
\begin{align}\label{nxrho}
\nonumber E N_{x_i,\rho}&\leq \sum_{m=0}^\infty {{i+2m}\choose {m}} \lambda^{i+2m} d^m \leq \lambda^i(1+ \sum_{m=1}^\infty 2^{i+2m} \lambda^{2m} d^m) \\
& = \lambda^i (1+2^i\sum_{m=1}^\infty (4 \lambda^2 d)^m) \leq (1+\eta)\lambda^i.
\end{align} 
Since $\lambda^2 d\to 0$ as $n\to \infty$, $\eta>0$ can be arbitrarily small if $n$ is large enough. It follows that 
\beq\label{stationary}
P(x_i\in \xi_\infty) \leq P(N_{x_i,\rho}\geq 1) \leq E (N_{x_i,\rho})\leq (1+\eta)\lambda^i.
\eeq

\subsection{Proof of Theorem \ref{ubsurv}}

\begin{proof}
Call the subgraph consisting of $\rho$ and all of its neighbors the \textit{central star}. We observe that the simultaneous occurrence of the following two events will lead to the extinction of the contact process on $S_k$: 
\begin{align*}
 G &=\{ \text{all the particles on the central star die before they give birth}\}, \\
 B^c&=\{ \text{no particle from outside the central star recolonizes the root} \}.
 \end{align*} 

We start by estimating $P(G)$. Starting with all sites occupied on $S_k$, we will set the center $\rho$ to be occupied for a certain amount of time $M$ while the distribution of the contact process on each $T^*_{k,i}$ becomes close to the equilibrium $\xi_\infty$. The contact process $\xi_t$ on $T^*_k$ is additive so we can write $\xi_t = \xi^\rho_t \cup \hat{\xi}^1_t$ where $\hat{\xi}^1_t$ is the contact process on $T_k$ with initially all 1's, and $\xi^\rho_t$ the contact process on $T^*_k$ with $\rho$ initially and permanently occupied. By the time $\hat{\xi}^1_t$ dies out, we have $\xi_t=\xi^\rho_t$, whose distribution is stochastically dominated by the stationary distribution $\xi_\infty$.

For our purpose $M$ should be chosen to be roughly the extinction time of $\hat{\xi}^1_t$. To simplify notation we let $d=Cn^{1-\delta}$ and consider contact process on the regular tree $\TT_d$. Let $A_t$ be the contact process on $\TT_d$ with birth rate $\gamma$ and death rate 1. Since $T_k\subset \TT_d$, if we take $\gamma>\lambda$ the contact process $\xi_t$ on $T_k$ is stochastically dominated by $A_t$. Following the proof of Theorem 4.1 in part I of Liggett \cite{Lig99}, we define 
$$w_\theta(A_t)=\sum_{y\in A_t} \theta^{\ell(y)}$$
where $\ell(y)$ is the distance from the root of $\TT_d$ to $y$. Liggett shows that if $\theta = 1/\sqrt{d}$
$$
\left. \frac{d}{dt} E_A w_\theta(A_t) \right|_{t=0} \le [2\sqrt{d}\gamma -1] w_\theta(A_0)
$$
Taking $\gamma = 1/4\sqrt{d} \approx n^{-(1-\delta)/2 }\gg \lambda$ for large $n$, we have
$$
E_{T_k} w_\theta(A_t) \le w_{\theta}(T_k) e^{-t/2} \leq (d\theta)^k e^{-t/2}.
$$
Markov's inequality implies that $
d^{-k/2} P_{T_k}( \ell(x) \le k \hbox{ for some $x\in A_t$}) \le  d^{k/2}e^{-t/2}.$
It follows that
\beq\label{survtime}
P( \hat{\xi}^1_t \neq \emptyset ) \le P_{T_k}( \ell(x) \le k \hbox{ for some $x\in A_t$}) \le d^{k}e^{-t/2},
\eeq
so the process dies out with high probability when $t \geq 4k \log d$. Since $d=Cn^{1-\delta}$  we can choose $M=4k\log n$.

Now we start the contact process on $T^*_k$ with all sites occupied. We set $\rho$ to be occupied for the first $M$ units of time and then allow $\rho$ to become vacant at rate 1. For any $\eta>0$ and $t\geq M$, when $n$ is large enough
\begin{align*}
P(\xi_t(x_i)=1)&\leq P(\xi^\rho_t(x_i)=1)+P(\hat{\xi}^1_t(x_i)=1)
\leq P(x_i \in \xi_\infty)+P( \hat{\xi}^1_t \neq \varnothing)\\
&\leq (1+\eta)\lambda^i+n^{-(1-\delta)k}\leq (1+2\eta)\lambda^i \quad \quad \text{by (\ref{stationary})}.
\end{align*}

The probability that an occupied site adjacent to $\rho$  dies out before giving birth is $\geq 1/(1+\lambda)$. Hence if there are $m$ occupied neighbors of  $\rho$ when it becomes vacant  then $P(G) \ge (1+\lambda)^{-m}$. Since the occupancy of sites adjacent to the root are independent events, (\ref{stationary}) and the law of large numbers implies that $m \le (1+2\eta)\lambda n$ with high probability. From this it follows that for large $n$ 
\beq
P(G) \geq (1-\eta) (1+\lambda)^{- (1+2\eta)\lambda n}\geq (1-\eta)e^{-(1+2\eta)\lambda^2 n}. \label{PofG}
\eeq

Turning our attention to event $B^c$, we begin by noting that the expected number of particles outside the central star at time $t\geq M$ is
$$
\leq (1+2\eta)\left(na_1\lambda^2 + na_1a_2 \lambda^3 + \ldots n a_1 a_2 \cdots a_{k-1} \lambda^{k}\right).
$$
Since  $\lambda = \sqrt{(c \log n)/n}$, if  $ a_i = Cn^{1-\delta}$ for all $i$ this grows rapidly as $n\to\infty$. 

Fortunately, if we start a contact process on $S_k$ from a site on level $i$ and freeze any particle that reaches the center $\rho$, then the expected number of such particles is $\le (1+\eta)\lambda^i$ by the same dual path argument as in (\ref{nxrho}). Therefore the expected number $N_\rho$ of particles reaching the center $\rho$ is
$$
\le (1+\eta)(1+2\eta) \left(na_1 \lambda^4 + na_1a_2 \lambda^6 + \ldots n a_1 a_2 \cdots a_{k-1} \lambda^{2k+2}\right)\leq n^{-\delta/2}.
$$
Hence when $n$ is large
$$
P(B) \le P(N_\rho\geq 1)\leq E(N_\rho) \leq n^{-\delta/2} \leq \eta.
$$
$G$ and $B^c$ are both decreasing events, i.e., having more births or fewer deaths is bad for them, so by the Harris-FKG inequality 
$$
P(G \cap B^c) \ge P(G)P(B^c)\geq (1-\eta)^2 e^{-(1+2\eta)\lambda^2 n}.
$$
Note that when $G\cap B^c$ occurs the process dies out on $S_k$. When  $\rho$ first becomes vacant  after time $M$,  if $G\cap B^c$ occurs then we terminate the process and obtain an upper bound for the survival time; if $G\cap B^c$ does not occur then we set all of the site in $S_k$ to be 1 and make $\rho$ occupied for he next $M$ units of time.  Since $P(G\cap B^c)\geq (1-\eta)^2 e^{-(1+2\eta)\lambda^2 n}$, in expectation we need to try $(1-\eta)^{-2}e^{(1+2\eta)\lambda^2 n}$ times to have a success. It takes an expected amount of time $M+1$ for $\rho$ to become vacant. Let $Exp(\beta)$ denote an exponential random variable with intensity $\beta$. The expected amount of time needed to determine if $G$ occurs is at most
$$
\sum_{i=1}^n E(Exp(i))=\sum_{i=1}^n (1/i)\leq 2\log n.
$$
Let $\{X_i\}_{i=1}^n$ be a set of i.i.d. random variables representing the survival time of the contact process $\hat{\xi}^1_t$ on $\{T_{k,i}\}_{i=1}^n$, respectively. It takes an extra time $\leq E(\max_{1\leq i\leq n} X_i)$ to determine if $G\cap B^c$ occurs. Since 
\begin{align*}
E(\max_{1\leq i\leq n} X_i)&\leq 2(k+1)\log n +\int_{2(k+1)\log n}^\infty P(\max_{1\leq i\leq n} X_i>t)\;dt\\
&\leq 2(k+1)\log n+ \int_{2(k+1)\log n}^\infty n \cdot (Cn^{1-\delta})^ke^{-t/2} \;dt\leq 4k\log n \quad\quad \text{ by (\ref{survtime})},
\end{align*}
each round takes at most $M+1+2\log n+4k\log n\leq 9k\log n$ units of time in expectation. It follows that 
$$E \tau_k \leq (9k\log n)(1-\eta)^{-2}e^{(1+2\eta)\lambda^2 n}\leq C_0(\log n)e^{(1+2\eta)\lambda^2 n}$$
for some $C_0>0$ and sufficiently large $n$.
\end{proof}

\section{Lower bound on $\lambda_2$ }

\begin{lemma} \label{lbcrab}
Let $c_k=(k-b)/2$ and $\epsilon>0$. When $n$ is sufficiently large, the critical value $\lambda_2$ of the contact process on the $(n,a_1,a_2,\dots,a_k)$ tree in Theorem \ref{akub2} satisfies
$$
\lambda_2\geq \sqrt{\frac{c_k\log n}{(1+\ep)n}}.
$$
\end{lemma}

\begin{proof}
Let $S(\rho)\equiv S_k\cup L_{k+1}$, where we recall that $L_{k+1}=\{ y: d(y,\rho)=k+1\}$. We start the contact process on $S(\rho)$. When a site in $S_k$ gives birth onto a site on $L_{k+1}$  we freeze the particle at $y_{k+1}$
We begin with only $\rho$ occupied and run the process until there are no particles on $S_k$. These particles will be the descendants of the initial particle $\rho$ in a branching random walk that we use to dominate the contact process on the periodic tree. When the contact process on $S_k$ dies out we are left with frozen particles in $L_{k+1}$. Each frozen particle at $y_{k+1}$ starts a new contact process on a subgraph $S(y_{k+1})\equiv \{ z: |z-y_{k+1}|\leq k+1\}$ which is isomorphic to $S(\rho)$ and has center $y_{k+1}$. If there are several frozen particles at the same site they start independent contact processes. Then we freeze every particle that escapes from $S(y_{k+1})$, and so on. 

Let $B(S(\rho),y_{k+1})$ be the total number of particles frozen at $y_{k+1} \in L_{k+1}$ in the contact process on $S(\rho)$. Let $y_1$ denote the neighbor of $\rho$ that is at distance $k$ to $y_{k+1}$. When the center $\rho$ is occupied, it gives birth to a particle at $y_1$ at rate $\lambda$. By the same reasoning as in (\ref{nxrho}), starting from a particle at $y_1$, to produce a particle at $y_{k+1}$ we need a path from $y_1$ to $y_{k+1}$ where in each step a birth occurs before death. If we set the center $\rho$ to be always occupied, then we can ignore the paths from $y_1$ to $y_{k+1}$ that go
through $\rho$. Therefore, starting from a particle at $y_1$ the expected number of particles reaching $y_{k+1}$ is $\leq (1+\eta)\lambda^k$ by the same computation as (\ref{nxrho}).

By Lemma \ref{ubsurv} the expected survival time on $S_k$ is $\leq C(\log n)e^{(1+\eta)\lambda^2 n}$. If during this whole time the center $\rho$ is occupied and pushing particles to $y_1$ at rate $\lambda$, then there are an expected number of 
$$\leq \lambda \cdot C(\log n)e^{(1+\eta)\lambda^2 n}$$
times we start a process from a particle at $y_1$ to produce particles at $y_{k+1}$. Hence 

$$
EB(S(\rho),y_{k+1}) \leq \lambda\cdot C(\log n) e^{(1+\eta)\lambda^2 n} \cdot (1+\eta)\lambda^k \leq \lambda^{k+1} n^{(1+2\eta)c}.
$$
To  bound the number of particles on the periodic tree that reach the root $\rho$, we consider a tree consisting of the vertices of degree $n$ in which each vertex is connected to the others vertices of degree $n$ at distance $k+1$. This is a $N$-regular tree with $N = n(a_1a_2\cdots a_k)$. Starting from the root, there are 
$$
\le {2m \choose m}  \, N^{m} \cdot 1^m \le 2^{2m} N^m
$$ 
paths of length $2m$ that returns to it. So the expected number of particles returning to the root is
\beq\label{sum}
\le 2^{2m} N^m (\lambda^{k+1} n^{(1+2\eta)c})^{2m} = (2N^{1/2} \lambda^{k+1} n^{(1+2\eta)c})^{2m}
\eeq
We have $\log(\lambda)/\log n \to -1/2$ and $(\log N)/\log n \to 1+b$ so
$$
\lim_{n\to\infty} \frac{\log(2 \lambda^{k+1} N^{1/2} n^{(1+\ep)c})}{\log n}
=-\frac{k+1}{2}+\frac{1+b}{2}+(1+2\eta)c<0
$$ 
if $c < (k-b)/2(1+2\eta)$. In this case, the expected number of particles that return to the origin is finite, which means the process does not survive locally. Taking $\ep=2\eta$ completes the proof.
\end{proof}

\section{Lower bound on survival time on stars} 

Here, following the approach of Chatterjee and Durrett \cite{ChaDur}, we will reduce the contact process on a star to a one dimensional chain.
We denote the state of the star by $(j,k)$ where $j$ is the number of occupied leaves and $k=1,0$ when the center is occupied, vacant. 
We will only look at times when the center is occupied. When the center is vacant and there are $j$ occupied leaves,
the next event will occur after exponential time with mean $1/(j\lambda + j)$. The
probability that it will be a birth at the center is $\lambda/(\lambda+1)$. 
The probability it will be the death of a leaf particle is $1/(\lambda+1)$.
Thus, the number of leaf particles $Z$ that will be lost while the center is vacant
has a shifted geometric distribution with success probability $\lambda/(\lambda+1)$, i.e.,
\beq
 P(Z=j) = \left( \frac{1}{\lambda+1}\right)^{j} \cdot \frac{\lambda}{\lambda+1}
 \quad\hbox{for $j\ge 0$}.
\label{Zdist}
\eeq
Note that $EZ=1/\lambda$.
Since we are interested in a lower bound on the survival time, we can simply ignore the time spent when the center is vacant. Here we will construct a process $X_t$ that gives a lower bound on the number of occupied leaves in the contact process.

Let $\delta>0$ and $L=(1-4\delta)\lambda n$. When there are $k\leq L$ occupied leaves and the center is occupied, new leaves become occupied at rate 
$$
\lambda(n-k) \ge \lambda(n -\lambda n) \ge \lambda(1-\delta)n
$$ 
for sufficiently large $n$ since $\lambda=\sqrt{c\log n/n} \to 0$ as $n\to\infty$.

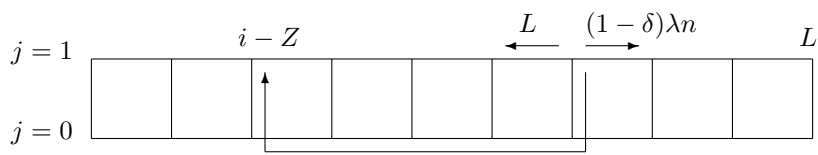
\begin{center}
\begin{picture}(330,60)
\put(0,40){$j=1$}
\put(0,10){$j=0$}
\put(30,10){\line(1,0){270}}
\put(30,40){\line(1,0){270}}
\put(30,10){\line(0,1){30}}
\put(60,10){\line(0,1){30}}
\put(90,10){\line(0,1){30}}
\put(120,10){\line(0,1){30}}
\put(150,10){\line(0,1){30}}
\put(180,10){\line(0,1){30}}
\put(210,10){\line(0,1){30}}
\put(240,10){\line(0,1){30}}
\put(270,10){\line(0,1){30}}
\put(300,10){\line(0,1){30}}
\put(295,45){$L$}
\put(215,45){\vector(1,0){20}}
\put(215,50){$(1-\delta)\lambda n$}
\put(205,45){\vector(-1,0){20}}
\put(190,50){$L$}
\put(215,35){\line(0,-1){30}}
\put(215,5){\line(-1,0){120}}
\put(95,5){\vector(0,1){30}}
\put(85,45){$i-Z$}
\end{picture}
\end{center}

Let $X_t$ have the following transition rates:
\begin{center}
 \begin{tabular}{lc}
 jump & at rate \\
$ X_t \to X_t-1$ & $L$ \\
$ X_t \to \min\{X_t+1, L\}$  & $(1-\delta)\lambda n$ \\
$ X_t \to X_t-Z$ & $1$
\end{tabular}
\end{center} 
Here $Z$ is independent of $X_t$ and has the distribution given in \eqref{Zdist}.

\begin{lemma} \label{super}
Let $\delta>0$. Suppose $\lambda = \sqrt{c(\log n)/n}$ and let 
$$
\theta=\frac{1}{\lambda+1}\left(  \lambda - \frac{1}{\delta\lambda n } \right)
$$ 
If $n$ is large then $h(X_t) \equiv (1-\theta)^{X_t}$ is a supermartingale when $X_t < L$.
\end{lemma}

\begin{proof}
Suppose the current value is $V = (1-\theta)^{X_t}$ where $X_t \le L=(1-4\delta)\lambda n$. We have
\begin{align*}
V \to V/(1-\theta) &\quad\hbox{at rate $\le L$} \\
V \to V(1-\theta) &\quad\hbox{at rate $\ge (1-\delta) \lambda n$} \\
V \to V(1-\theta)^{-Z} & \quad\hbox{at rate 1}
\end{align*}

The changes in value due to the first two transitions are, if $\theta$ is small,
\begin{align*}
V \left(\frac{1}{1-\theta} - 1\right) \le (1-\delta)^{-1}\theta V &\quad\hbox{at rate $\le L$} \\
V[(1-\theta) -1] = -\theta V &\quad\hbox{at rate $\ge (1-\delta)\lambda n$} 
\end{align*}
We have $L =  (1-4\delta)\lambda n <(1-\delta)(1-3\delta)\lambda n$, so the first two types of jumps have a net drift 
\beq
 \left((1-\delta)^{-1}L - (1-\delta)\lambda n\right) \theta V \le  -(2\delta \lambda n ) \theta V.
\label{part1}
\eeq
In the third case, ignoring the fact that the number of occupied leaves cannot drop below 0, we have
\begin{align*}E(1-\theta)^{-Z} 
& \le \sum_{k=0}^\infty \left(\frac{1}{1+\lambda}\right)^{k} \frac{\lambda}{1+\lambda} \cdot (1-\theta)^{-k} \\
& = \frac{\lambda}{1+\lambda} \sum_{k=0}^\infty  \left(\frac{1}{(1+\lambda)(1-\theta)}\right)^{k} \\
& = \frac{\lambda}{1+\lambda} \cdot \frac{1} { 1 - \frac{1}{(1+\lambda)(1-\theta)}}  = \frac{\lambda(1-\theta)}{\lambda - \theta - \theta\lambda}
\end{align*}
so we have 
$V(E(1-\theta)^{-Z} - 1) = \frac{\theta V}{\lambda - \theta(1+\lambda)} =(\delta\lambda n)\theta V$
for the chosen value of $\theta$. Combining this with \eqref{part1}
gives that for any $\delta>0$, $h(X_t)$ is a supermartingale for large $n$.
\end{proof} 

We use $P_{i}$ to denote the law of the process $X_t$ starting with $X_0=i$. Since $X_t$ omits some time intervals from the contact process
on the star, the next result implies Theorem \ref{lbsurv}.

\begin{lemma}\label{timebd}
Let $L =(1-4\delta)\lambda n$. If $\eta>0$ is small then
$$
P_{L-1} \left( T_{\eta L}^- \ge \frac{1}{\lambda^2 n} e^{(1-4\eta)\lambda^2 n} \right) \to 1 \quad \text { as } n\to\infty.
$$
\end{lemma}

\begin{proof}
Suppose $a < x < b$ are integers. Let $T_a^- = \inf\{ t : X_t < a\}$, let $T_b = \inf\{ t : X_t=b\}$ and note that
$X(T_b)=b$ while $X(T^-_a) \le a-1$. Since $h(X_t)$ is a supermartingale
and $h$ is decreasing
$$
h(x) \ge h(a-1) P_x(T^-_a < T_b ) + h(b) [1 -  P_x(T^-_a < T_b ) ]
$$
Rearranging we have
$$
P_x(T^-_a < T_b ) \le \frac{h(x) -h(b)}{h(a-1)-h(b)}.
$$
When $x=b-1$ this implies
\begin{align*}
P_x(T^-_a < T_b ) & \le \frac{h(b-1) -(1-\theta) h(b-1)}{h(a-1)-h(b-1)}\\
& = \frac{\theta h(b-1)/h(a-1)}{1 - h(b-1)/h(a-1)} 
\end{align*} 

Let $\eta>0$. We will apply this result with $b = (1-4\delta)\lambda n$ and $a=\eta b$. 
If $\delta$ is small $b \ge (1-\eta)\lambda n$. If $\lambda$ is small then
$1-\theta < 1-(1-\eta)\lambda$.  With these choices 
$$
h(b-1)/h(a-1) = (1-\theta)^{b-a} < (1-(1-\eta)\lambda)^{(1-2\eta)\lambda n}\leq \exp\left(-(1-3\eta)\lambda^2n\right).
$$
If $n$ is large,
\beq
P_{b-1}( T^-_a < T_b ) \le  2 \lambda \exp(-(1-3\eta)\lambda^2 n).
\label{abub}
\eeq
Let  $G_L = \{ X_t\hbox{ returns }(1/2\lambda)e^{(1-4\eta)\lambda^2 n}\hbox{ times to $L$ before going } <\eta L \}.$
It follows from \eqref{abub} that
$P(G_L) \ge 1 - e^{-\eta\lambda^2 n}$.
In order to return to $L$ we have to jump from $L-1$ to $L$, a time that dominates an exponential random variable with parameter $\lambda n/2$ so the law of large numbers tells us that the total amount of time before $X_t < \eta L$ is 
$\ge \frac{1}{\lambda^2 n} e^{(1-4\eta)\lambda^2 n}$
on $G_L$ which completes the proof. 
\end{proof}

\section{Ignition on a star graph} 
In this section we will describe a mechanism for the contact process on the period tree $(n,a_1,\dots, a_k)$ to survive, which basically relies on the dynamics on the star graphs of degree $n$ embedded in tree. Vertices with degree $n$ will be called \textit{hubs}. 

There are three ingredients in the proof of the upper bounds on $\lambda_2$ 
\begin{enumerate}
  \item survival of the process on the star graph containing a hub for a long time, 
  \item pushing particles from one hub to other hubs at distance $k+1$, 
  \item  ``ignition'', which refers to increasing the number of occupied leaves at the new hub  to $L$.
\end{enumerate}  

\noindent
The first point was taken care of in the previous section. The third is covered in this one.
Starting from only the central vertex of a degree $n$ hub occupied, we need to increase the number of occupied leaves to $L=(1-4\delta)\lambda n$ by time $n^c/4$, which is referred to as the \textit{ignition} of a hub. We treat $L$ and $K$ in the following lemma as integers for simplicity.

\begin{lemma} \label{ignite}
Suppose $\lambda = \sqrt{c_0 (\log n)/n }$.
Let $T_{0,0}$ be the first time the star is vacant and $T_i$ be the first time the star has $i$ occupied leaves. For any small $\delta>0$ if $K = \lambda n/\sqrt{\log n}$
and $L=(1-4\delta)\lambda n$, then for large $n$
\begin{align*}
& (i) P_{0,1}( T_K > T_{0,0} )  \le  3/\sqrt{\log n},  \\
& (ii) P_{K,1} (T_{0,0} < T_L)  \le  2\exp(- (c_0/3) \sqrt{\log n}) \\
& (iii) E_{0,1} \min\{ T_{0,0},  T_L \} \le (1 + \log n)/2\delta
\end{align*}
\end{lemma}

\begin{proof}  
Let $p_0(t)$ be the probability a leaf is occupied at time $t$ when there are no occupied leaves at time 0 and
the central vertex has been occupied for all $s \le t$. $p_0(0)=0$ and
$$
\frac{dp_0(t)}{dt} = - p_0(t) + \lambda(1-p_0(t)) = \lambda -(\lambda+1) p_0(t) 
$$
Solving gives 
$p_0(t) = \lambda ( 1 - e^{-(\lambda +1)t} )/(\lambda+1)$.
As $t \to 0$, $p_0(t) \sim \lambda t$
so if $t$ is small $p_0(t) \ge \lambda t/2$

Taking $t = 2/\sqrt{\log n}$ it follows that if $B =  \hbox{Binomial}(n,\lambda/\sqrt{\log n})$
$$
P_{0,1}( T_K < T_{0,0} ) \ge P( B > K) \exp(-2/\sqrt{\log n})
$$
The second factor is the probability that the center stays occupied until time $2/\sqrt{\log n}$, and 
$\exp(-2/\sqrt{\log n})\ge 1 - 2/\sqrt{\log n}. $
$B$ has mean $\lambda n/\sqrt{\log n}$ and variance $\le \lambda n/\sqrt{\log n}$ so Chebyshev's inequality implies
$$
P( B < \lambda n/(2\sqrt{\log n}) ) \le \frac{\lambda n/\sqrt{\log n}}{(\lambda n/(2\sqrt{\log n}))^2} \le \frac{4 \sqrt{\log n}}{\lambda n} \le \frac{1}{\sqrt{\log n}}.
$$
For (ii) we use the supermartingale $h(X_t)$ from Lemma \ref{super}. 
\begin{align*}
P_{K,1}(T_{0,0} < T_L) & \le 2(1-\lambda/3)^{\lambda n/\sqrt{\log n}}\\
& \le 2\exp(- \lambda^2 n/3\sqrt{\log n} ) =2 \exp( - (c_0/3) \sqrt{\log n} ).
\end{align*}

For (iii) we compare with the process $X_t$ in which we ignore the time spent when the center is vacant. To bound the time for the process $X_t$ to reach $L$ or die out we note that $EZ = (\lambda+1)/\lambda - 1 = 1/\lambda$ so when $n$ is large
$$
\mu =(1-\delta) \lambda n- (1-4\delta)\lambda n- 1/\lambda=3\delta \lambda n-1/\lambda \geq 2\delta \lambda n
$$
gives a lower bound on the drift. Let $\hat{T}_{0,0}$ be the first time $X_t$ hits 0 and $\hat{T}_L$ be the first time $X_t$ hits $L$.
$X_t - \mu t$ is a submartingale before time $V_L = \hat{T}_{0,} \wedge \hat{T}_L$.  Stopping the submartingale $X_t-\mu t$ at the bounded stopping time $V_L \wedge s$ 
$$
EX(V_L\wedge s) - \mu E(V_L\wedge s) \ge EX_0=0.
$$
Since $EX(V_L\wedge s) \le L$, it follows that
$E(V_L\wedge s)\le L/\mu$.
 
Letting $s\to\infty$ we have $EV_L\le L/\mu \le 1/2\delta $ since $L=(1-4\delta)\lambda n$ and $\mu\geq 2\delta \lambda n$.
Note that the above calculation is for $X_t$ which ignores the time when the center is vacant. To bound the time when the center is vacant, we note that
the most extreme excursion that starts at $n$ and goes to 0 takes a time with mean $(\log n)/(1+\lambda)$. During time $[0, V_L]$ the excursions occur at rate 1, that is,
$E_{0,1} \min\{ T_{0,0},  T_L \} \le (1+\log n)E V_L \leq (1+\log n)/2\delta.$
\end{proof}

\section{Upper bound on $\lambda_2$}

We first prove the result for the $(1,n)$ tree. Suppose that $\lambda = \sqrt{c (\log n)/n}$ with $c > 1/2$.  We select one hub to call the root. In the first phase of the construction we only allow the particles to be passed to other hubs that are children of the current vertex and are at distance 2.

\mn
{\bf Step 1. Pushing the particles out to distance 2m.} 

\mn
Lemma \ref{timebd} implies that if there at least $L$ occupied leaves before time $n^c/4$ then with high probability we have 
$$
G_0 = \{ \hbox{there will be at least $\eta L$ occupied leaves during $[n^c/4, 3n^c/4]$} \}. 
$$
During this time interval the hub will try to push particles to hubs at distance $2$. The first step is for the leaves to give birth onto the center. 

\begin{lemma}
Suppose that the number of occupied leaves is always $\ge \eta L$ on $[0,n^c/2]$.
Let $t_0=2/(1-4\delta)\eta$ and $G_1 = \{$there is no interval of length $\ge t_0$  
during which the center is always vacant$\}$. As $n\to\infty$, $P(G_1) \to 1$.
\end{lemma}

\begin{proof}
The center becomes vacant at times of a Poisson process with rate 1. Using large deviations results for the Poisson process, the probability there are more than $n^c$ arrivals in an interval of length $n^c/2$ is $\le \exp(-\gamma n^c)$. 
Suppose the center is vacant and  let $R$ be the time needed until it becomes occupied. 
$$
P( R > t_0 ) \le \exp(-t_0\lambda \eta L) = \exp(-2c\log n) = n^{-2c}
$$ 
Hence $P(G^c_1) \le \exp(-\gamma n^c) + n^c n^{-2c} \to 0$.
\end{proof}

When the center is occupied  there is probability $e^{-1}(1-e^{-\lambda})$ that it will stay occupied for time 1 and give birth onto a given leaf within time 1. With probability $e^{-1}$ that leaf will stay occupied until time 1. Doing this twice the probability of passing a particle to a given hub at distance 2 is
$$
\ge \left[ e^{-1} (1-e^{-\lambda}) e^{-1} \right]^{2} \ge \lambda^{2}/4e^{4} \ge \frac{C_1 \log n}{n} 
\qquad\hbox{where $C_1=c/4e^{4}$.}
$$
Since our cycle takes time $t_1\equiv t_0+2$, we have $n^c/t_1$ chances to do this during $[n^c/4, 3n^c/4]$. The probability that all attempts fail is
\beq
\le \left( 1- \frac{C_1 \log n}{n} \right)^{n^c/t_1} \le  \exp\left( - C_2 n^{c-1} \log n \right) \quad\text{ where }C_2=\frac{C_1}{t_1}.
\label{push}
\eeq

\medskip
It follows from \eqref{push} that  the probability of a successful push in $[n^c/4, 3n^c/4]$ is 
\beq\label{pushc}
\geq 1-\exp\left( - C_2 n^{c-1} \log n \right)\geq n^{c-1} \quad\hbox{when $n$ is large. }
\eeq
We say that a hub at distance $2m$ that is a descendant of the root is \textit{wet}  if it has $\ge \eta L$ occupied leaves at time $mn^c$. Starting with a wet hub at time 0, the center of the hub will become occupied at rate at least $\eta L$ and then Lemma \ref{ignite} implies that the hub will be ignited within time $n^c/4$ with high probability. When the hub successfully pushes a particle to an adjacent hub during time $[n^c/4, 3n^c/4]$, that adjacent hub can ignite within the next $n^c/4$ units of time with high probability and hence be wet at time $n^c$. Therefore a wet hub can make an adjacent hub wet with probability $\geq n^{c-1}$ when $n$ is large.
If we condition on $G_0 \cap G_1$ then the pushing events for different neighbors are independent,  the number of neighboring hubs that become wet is $\ge \hbox{binomial}(n,n^{c-1})$. Let $Z_m$ be the number of wet hubs at distance $2m$ at time $mn^c$, then $EZ_m = (n^{c})^m$. 

\mn
{\bf Step 2. Bringing a particle back to the root.} 

\medskip
Each star at distance $2m$ from the root has a unique path back to the root.  By \eqref{pushc}
the probability for pushing a particle at each step is $\ge n^{c-1}$ so the mean number of paths $N_m$ that go out a distance 
$2m$ and lead back to the origin is
$$
EN_m \ge n^{(2c-1)m}.
$$ 
The different paths back to the root are not independent.
The number of paths from distance $2m$ back to the root that agree in the last $k$ steps is 
$\sim n^k (n^{m-k})^2.$
The probability that all the edges in the combined path are successful pushes is 
$$
n^{2(c-1)[k + 2(m-k)]}.
$$
Here and in the next step we are assuming that the probability of a successful push is exactly $n^{c-1}$
The expected number of successful pairs of paths outward to distance $2m$ that agree in the first $k$ steps is
$\sim  n^{(2c-1)[k + 2(m-k)]}.$
Thus the second moment of the number of successful paths out and back is
$$
EN^2_m  \le \sum_{k=0}^m n^{(2c-1)[k + 2(m-k)]} 
\le n^{(2c-1)2m} \left( 1 + \sum_{\ell=1}^m n^{-(2c-1) \ell} \right).
$$
This implies that $(EN_m)^2/E(N_m^2) \to 1$. The Cauchy-Schwarz inequality implies that
$$
E\left( N_m 1_{\{N_m>0\}} \right)^2 \le E(N_m^2) P(N_m>0).
$$
Rearranging we conclude that 
$P(N_m>0) \ge (E N_m)^2 /E(N_m^2) \to 1.$ 
Since $m$ is arbitrary we have 
that a particle returns to the root at arbitrarily large times so we have established our upper bound on $\lambda_2$ for the $(1,n)$ tree.

\mn
{\bf Step 3: Proof for $(n,a_1, \ldots, a_k)$ tree}

\medskip
The structure of the proof is almost the same as the period 2 case given above, so we content ourselves
to compute the constant. The probability of pushing a particle to a given hub at distance $k+1$ is 
$$\geq [e^{-1}(1-e^{-\lambda})e^{-1}]^{k+1}\geq C_1\lambda^{k+1} \quad \text{where }C_1=(\frac{1}{2e^2})^{k+1}.$$
Suppose $\lambda=\sqrt{(c\log n)/n}$. It follows from similar computations to (\ref{pushc}) that the probability of a successful push between time $[n^c/4,3n^c/4]$ is $\geq n^{c-(k+1)/2}$. Since there are $N=n(a_1a_2\cdots a_k)$ hubs at distance $k+1$ from the root, the expected number of particles we will generate per step as we work our way away from the root is $N n^{c-(k+1)/2}$ and as we work our way back in is $n^{c-(k+1)/2}$.
To guarantee a lineage returning to the root, i.e., $N_m>0$ for arbitrarily large $m\in \mathbb{N}$, we need 
$$\lim_{n\to\infty} \frac{ \log(N n^{c-(k+1)/2} \cdot n^{c-(k+1)/2})}{\log n} >0,$$
i.e.,$1+b+2c-(k+1)>0.$
So we want 
$c> (k+1)/2-(b+1)/{2}=(k-b)/2,$
the constant given in the theorem.







\end{document}